\documentclass[preprint,review,10pt]{amsart}
\usepackage{amsmath,amssymb,amsfonts,xspace}
\newtheorem{theorem}{Theorem}[section]

\theoremstyle{definition}
\newtheorem{definition}[theorem]{Definition}
\newtheorem{example}[theorem]{Example}

\newtheorem{proposition}[theorem]{Proposition}
\newtheorem{corollary}[theorem]{Corollary}
\newtheorem{lem}[theorem]{Lemma}

\theoremstyle{remark}

\numberwithin{equation}{section}

\begin{document}

\title{ On weakly $S$-primary hyperideals }

\author{Mahdi Anbarloei}
\address{Department of Mathematics, Faculty of Sciences,
Imam Khomeini International University, Qazvin, Iran.
}

\email{m.anbarloei@sci.ikiu.ac.ir}


\subjclass[2020]{  20N20, 16Y20  }


\keywords{  multiplicative set, weakly $n$-ary $S$-primary hyperideals, strongly weakly $n$-ary $S$-primary hyperideals.}

\begin{abstract}
 In this paper, our purpose is to introduce and study the notion of weakly $n$-ary $S$-primary hyperideals in a commutative Krasner $(m,n)$-hyperring.

\end{abstract}
\maketitle

\section{Introduction} 
Primary ideals are a significant concept in the study of commutative rings, providing a framework for comprehending the structure and properties of ideals in relation to prime ideals and other algebraic constructs. The notion of $S$-primary ideals extending  the concept of primary ideals in commutative rings by incorporating multiplicativ subsets,  was introduced and studied  in \cite{massaoud} and \cite{visweswaran}.
The notion of weakly $S$-primary ideals was proposed by Celikel and Khashan in \cite{Celikel}. The study often builds on previous work regarding weakly primary ideals and other related concepts.
Assume that $S$ is a multiplicative subset of a commutative ring $A$. An ideal $Q$ of $A$ is called weakly $S$-primary if there exists $s \in S$ such that for every $x,y \in A$ if $0 \neq xy \in Q$, then $sx \in Q$ or $sy \in rad(Q)$. Hyperstructures represent a natural extension of classical algebraic structures and they were
defined by the French mathematician F. Marty. In 1934, Marty \cite{s1} defined the concept of a hypergroup as a generalization of groups during the $8^{th}$ Congress of the Scandinavian Mathematicians. A comprehensive
review of the theory of hyperstructures can be found in \cite {amer2, s2, s3, davvaz1, davvaz2, s4, s10, jian}. The simplest algebraic hyperstructures which possess the properties of closure and associativity are called semihypergroups. 
$n$-ary semigroups and $n$-ary groups are algebras with one $n$-ary operation which is associative and invertible in a generalized sense. The notion of investigations of $n$-ary algebras goes back to Kasner’s lecture \cite{s5} at a scientific meeting in 1904. In 1928, Dorente wrote the first paper concerning the theory of $n$-ary groups \cite{s6}. Later on, Crombez and Timm \cite{s7, s8} defined and described the notion of the $(m, n)$-rings and their quotient structures. Mirvakili and Davvaz [20] defined $(m,n)$-hyperrings and obtained several results in this respect. In \cite{s9}, they introduced and illustrated a generalization of the notion of a hypergroup in the sense of Marty and a generalization of an $n$-ary group, which is called $n$-ary hypergroup.
The $n$-ary structures has been studied in \cite{l1, l2, l3, ma, rev1}. Mirvakili and Davvaz \cite{cons} defined $(m,n)$-hyperrings and obtained several results in this respect. It was Krasner, who introduced one important class of hyperrings, where the addition is a hyperoperation, while the multiplication is an ordinary binary operation, which is called Krasner hyperring. In \cite{d1}, a generalization of the Krasner hyperrings, which is a subclass of $(m,n)$-hyperrings, was defined by Mirvakili and Davvaz. It is called Krasner $(m,n)$-hyperring. Ameri and Norouzi in \cite{sorc1} introduced some important
hyperideals such as Jacobson radical, n-ary prime and primary hyperideals, nilradical, and n-ary multiplicative subsets of Krasner $(m, n)$-hyperrings. Afterward, the notions of $(k,n)$-absorbing hyperideals and $(k,n)$-absorbing primary hyperideals were studied by Hila et al. \cite{rev2}. 
Norouzi et al. proposed and analysed a new definition for normal hyperideals in Krasner $(m,n)$-hyperrings, with respect to that one given in \cite{d1} and they showed that these hyperideals correspond to strongly regular relations \cite{nour}. Asadi and Ameri introduced and studied direct limit of a direct system in the category of Krasner $(m,n)$-hyperrigs \cite{asadi}. 
Dongsheng defined the notion of $\delta$-primary ideals in a commutative ring where $\delta$ is a function that assigns to each ideal $I$ an ideal $\delta(I)$ of the same ring \cite{bmb2}. Moreover, in \cite{bmb3} he and his colleague investigated 2-absorbing $\delta$-primary ideals which unify 2-absorbing ideals and 2-absorbing primary
ideals. Ozel Ay et al. generalized the notion of $\delta$-primary on Krasner hyperrings \cite{bmb4}. The concept of $\delta$-primary hyperideals in Krasner $(m,n)$-hyperrings, which unifies the prime and primary hyperideals under one frame, was defined in \cite{mah3}. 

In this paper, the author aims to present the idea of weakly $n$-ary $S$-primary hyperideals in a commutative Krasner $(m,n)$-hyperring $A$ where $S$ is a multiplicative subset of $A$. Among many
results in this paper, we prove that if $Q$ is a weakly $n$-ary $S$-primary hyperideal of $A$, then $rad(Q)$ is an $n$-ary $S$-prime hyperideal of $A$ in Theorem \ref{bidi}.  In Theorem \ref{25}, we show that if $(Q : s)$  is a weakly $n$-ary primary hyperideal of $A$  for some $s \in S$, then $Q$ is a weakly $n$-ary $S$-primary hyperideal of $A$. Some characterizations of this concept are given on cartesian product of Krasner $(m,n)$-hyperrings  in Theorem \ref{cart1} and Theorem \ref{cart2}. We conclude that if $Q$ is a strongly weakly $n$-ary $S$-primary hyperideal of $A$ and   $g(Q^{(n)}) \neq 0$, then  $Q$ is an $n$-ary $S$-primary hyperideal of $A$ in Theorem \ref{28}. Furthermore, in Theorem \ref{29}, a type of Nakayama Lemma is considered for strongly weakly $n$-ary $S$-primary hyperideal of $A$. 
\section{Preliminaries}
Recall first the definitions and basic terms from the hyperrings theory. \\
Let $A$ be a non-empty set and $P^*(A)$ be the
set of all the non-empty subsets of $H$. Then the mapping $f : H^n \longrightarrow P^*(A)$
is called an $n$-ary hyperoperation and the algebraic system $(A, f)$ is called an $n$-ary hypergroupoid. For non-empty subsets $U_1,..., U_n$ of $H$, define $f(U^n_1) = f(U_1,..., U_n) = \bigcup \{f(x^n_1) \ \vert \ a_i \in U_i, i = 1,..., n \}$
The sequence $x_i, x_{i+1},..., x_j$ 
will be denoted by $x^j_i$. For $j< i$, $x^j_i$ is the empty symbol. Using this notation,
$f(x_1,..., x_i, y_{i+1},..., y_j, z_{j+1},..., z_n)$
will be written as $f(x^i_1, y^j_{i+1}, z^n_{j+1})$. The expression will be written in the form $f(x^i_1, y^{(j-i)}, z^n_{j+1}),$ when $y_{i+1} =... = y_j = y$. 
If for every $1 \leq i < j \leq n$ and all $x_1, x_2,..., x_{2n-1} \in A$, 
$f\bigg(x^{i-1}_1, f(x_i^{n+i-1}), x^{2n-1}_{n+i}\bigg) = f\bigg(x^{j-1}_1, f(x_j^{n+j-1}), x_{n+j}^{2n-1}\bigg)$
then the n-ary hyperoperation $f$ is called associative. An $n$-ary hypergroupoid with the
associative $n$-ary hyperoperation is called an $n$-ary semihypergroup. 
An $n$-ary hypergroupoid $(A, f)$ in which the equation $b \in f(a_1^{i-1}, x_i, a_{ i+1}^n)$ has a solution $x_i \in A$
for every $a_1^{i-1}, a_{ i+1}^n,b \in A$ and $1 \leq i \leq n$, is called an $n$-ary quasihypergroup, when $(A, f)$ is an $n$-ary
semihypergroup, $(A, f)$ is called an $n$-ary hypergroup. 
An $n$-ary hypergroupoid $(A, f)$ is commutative if for all $ \sigma \in \mathbb{S}_n$, the group of all permutations of $\{1, 2, 3,..., n\}$, and for every $a_1^n \in A$ we have $f(a_1,..., a_n) = f(a_{\sigma(1)},..., a_{\sigma(n)})$.
If $a_1^n \in A$, then $(a_{\sigma(1)},..., a_{\sigma(n)})$ is denoted by $a_{\sigma(1)}^{\sigma(n)}$.
If $f$ is an $n$-ary hyperoperation and $t = l(n- 1) + 1$, then $t$-ary hyperoperation $f_{(l)}$ is given by
$f_{(l)}(x_1^{l(n-1)+1}) = f\bigg(f(..., f(f(x^n _1), x_{n+1}^{2n -1}),...), x_{(l-1)(n-1)+1}^{l(n-1)+1}\bigg).$
\begin{definition}
\cite{d1} A non-empty subset $K$ of an $n$-ary hypergroup $(A, f)$ is said to be 
an $n$-ary subhypergroup of $A$ if $(K,f)$ is an $n$-ary hypergroup.
An element $e \in A$ is called a scalar neutral element if $x = f(e^{(i-1)}, x, e^{(n-i)})$, for every $1 \leq i \leq n$ and
for every $x \in A$. 

An element $0$ of an $n$-ary semihypergroup $(A, g)$ is called a zero element if for every $x^n_2 \in A$, 
$g(0, x^n _2) = g(x_2, 0, x^n_ 3) = ... = g(x^n _2, 0) = 0$.
If $0$ and $0^ \prime $ are two zero elements, then $0 = g(0^ \prime , 0^{(n-1)}) = 0 ^ \prime$ and so the zero element is unique. 
\end{definition}
\begin{definition}
\cite{l1} Let $(A, f)$ be a $n$-ary hypergroup. $(A, f)$ is called a canonical $n$-ary
hypergroup if:\\
(1) There exists a unique $e \in A$, such that for every $x \in A, f(x, e^{(n-1)}) = x$;\\
(2) For all $x \in A$ there exists a unique $x^{-1} \in A$, such that $e \in f(x, x^{-1}, e^{(n-2)})$;\\
(3) If $x \in f(x^n _1)$, then for all $i$, $x_i \in f(x, x^{-1},..., x^{-1}_{ i-1}, x^{-1}_ {i+1},..., x^{-1}_ n)$.

$e$ is said to be the scalar identity of $(A, f)$ and $x^{-1}$ is the inverse of $x$. Notice that the inverse of $e$ is $e$.
\end{definition}
\begin{definition}
\cite{d1} $(A, f, g)$, or simply $A$, is said to be a Krasner $(m, n)$-hyperring if:\\
(1) $(A, f$) is a canonical $m$-ary hypergroup;\\
(2) $(A, g)$ is a $n$-ary semigroup;\\
(3) The $n$-ary operation $g$ is distributive with respect to the $m$-ary hyperoperation $f$ , i.e., for every $a^{i-1}_1 , a^n_{ i+1}, x^m_ 1 \in A$, and $1 \leq i \leq n$,
\[g\bigg(a^{i-1}_1, f(x^m _1 ), a^n _{i+1}\bigg) = f\bigg(g(a^{i-1}_1, x_1, a^n_{ i+1}),..., g(a^{i-1}_1, x_m, a^n_{ i+1})\bigg);\]
(4) $0$ is a zero element (absorbing element) of the $n$-ary operation $g$, i.e., for every $x^n_ 2 \in R$ , 
$g(0, x^n _2) = g(x_2, 0, x^n _3) = ... = g(x^n_ 2, 0) = 0$.
\end{definition}
Throughout this paper, $(A, f, g)$ denotes a commutative Krasner $(m,n)$-hyperring.

A non-empty subset $B$ of $A$ is called a subhyperring of $A$ if $(B, f, g)$ is a Krasner $(m, n)$-hyperring. 
The non-empty subset $Q$ of $(A,f,g)$ is a hyperideal if $(Q, f)$ is an $m$-ary subhypergroup
of $(A, f)$ and $g(x^{i-1}_1, Q, x_{i+1}^n) \subseteq Q$, for every $x^n _1 \in A$ and $1 \leq i \leq n$.

\begin{definition} \cite{sorc1} Let $x$ be an element in a Krasner $(m,n)$-hyperring $A$. The hyperideal generated by $x$ is denoted by $<x>$ and defined as $<x>=g(A,x,1^{(n-2)})=\{g(r,x,1^{(n-2)}) \ \vert \ r \in A\}.$
Also, let $Q$ be a hyperideal of $A$. Then define $(Q:x)=\{a \in A \ \vert \ g(a,x,1_A^{(n-2)}) \in Q\}$. 
\end{definition}

\begin{definition} \cite{sorc1}
An element $x \in A$ is said to be invertible if there exists $y \in A$ such that $1_A=g(x,y,1_A^{(n-2)})$. Also,
the subset $U$ of $A$ is invertible if and only if every element of $U$ is invertible.
\end{definition}
\begin{definition}
\cite{sorc1} Let $Q \neq A$ be a hyperideal of a Krasner $(m, n)$-hyperring $A$. $Q$ is a prime hyperideal if for hyperideals $Q_1,..., Q_n$ of $A$, $g(Q_1^ n) \subseteq P$ implies that $Q_1 \subseteq Q$ or $Q_2 \subseteq Q$ or ...or $Q_n \subseteq Q$.
\end{definition}
\begin{lem} 
(Lemma 4.5 in \cite{sorc1}) Let $Q\neq A$ be a hyperideal of a Krasner $(m, n)$-hyperring $A$. Then $Q$ is a prime hyperideal if for all $x^n_ 1 \in A$, $g(x^n_ 1) \in Q$ implies that $x_1 \in Q$ or ... or $x_n \in Q$. 
\end{lem}
\begin{definition} \cite{sorc1} Let $Q$ be a hyperideal in a Krasner $(m, n)$-hyperring $A$ with
scalar identity. The radical (or nilradical) of $Q$, denoted by $rad(Q)$
is the hyperideal $\bigcap Q$, where
the intersection is taken over all prime hyperideals $Q$ which contain $Q$. If the set of all prime hyperideals containing $Q$ is empty, then $rad(Q)$ is defined to be $A$.
\end{definition}
Ameri and Norouzi showed that if $x \in rad(Q)$ then 
there exists $u \in \mathbb {N}$ such that $g(x^ {(u)} , 1_A^{(n-u)} ) \in Q$ for $u \leq n$, or $g_{(l)} (x^ {(u)} ) \in Q$ for $u = l(n-1) + 1$ \cite{sorc1}.
\begin{definition}
\cite{sorc1} Let $Q$ be a proper hyperideal of a Krasner $(m, n)$-hyperring $A$ with the scalar identity $1_A$. $Q$ is a primary hyperideal if $g(x^n _1) \in Q$ and $x_i \notin Q$ implies that $g(x_1^{i-1}, 1_A, x_{ i+1}^n) \in rad(Q)$ for some $1 \leq i \leq n$.
\end{definition}
If $Q$ is a primary hyperideal in a Krasner $(m, n)$-hyperring $A$ with the scalar identity $1_A$, then $rad(Q)$ is prime. (Theorem 4.28 in \cite{sorc1})
\begin{definition} \cite{sorc1}
Let $S \subseteq A$ be a non-empty subset of a Krasner $(m,n)$-hyperring $R$. $S$ is called an n-ary multiplicative if $g(s_1^n) \in S$ for $s_1,...,s_n \in S$.
\end{definition}

\section{weakly $n$-ary $S$-primary hyperideals}
In  \cite{mah6}, the notion of $n$-ary primary hyperideals was extended to $n$-ary $S$-primary hyperideals via an $n$-ary multiplicative subset  in a  Krasner $(m,n)$-hyperring. Let $A$ be a commutative Krasner $(m,n)$-hyperring, $S \subseteq A$ an $n$-ary multiplicative set   and $Q$  a hyperideal of  $A$ with $Q \cap S=\varnothing$. The hyperideal $Q$ refers to an $n$-ary $S$-primary hyperideal  if there exists an $s \in S$ such that for every $a_1^n \in A$ with $ g(a_1^n) \in Q$,  we have $g(s,a_i,1_A^{(n-2)}) \in Q$ or $g(a_1^{i-1},s,a_{i+1}^n) \in rad(Q)$ for some $i \in \{1,\cdots,n\}$.  Now, we  introduce and investigate  the concept of weakly $n$-ary $S$-primary hyperideals which constitutes a generalization of $n$-ary $S$-primary hyperideals. The weakly-version of $n$-ary $S$-primary hyperideals is defined as follows. 
\begin{definition}
Let $Q$ be a hyperideal of  $A$ and $S$ be an $n$-ary multiplicative subset of $A$ satisfying $Q \cap S=\varnothing$.  $Q$ refers to a weakly $n$-ary $S$-primary hyperideal  if there exists an $s \in S$ such that for all $a_1^n \in A$ with $0 \neq g(a_1^n) \in Q$, we have $g(s,a_i,1_A^{(n-2)}) \in Q$ or $g(a_1^{i-1},s,a_{i+1}^n) \in rad(Q)$ for some $i \in \{1,\cdots,n\}$.  
\end{definition}
\begin{example} \label{hanieh2}
Consider  the Krasner $(2,3)$-hyperring $(A=[0,1],\boxplus,\cdot )$ where the 2-ary hyperoperation $``\boxplus"$  is defined as $p \boxplus q=\{\max\{p, q\}\}$ if $p \neq q$  or $[0,p]$ if $p =q.$
Also, $``\cdot"$ is the usual multiplication on real numbers. Let $S=(0.0.1] \cup \{1\}$. Then  the hyperideal $Q=[0,0.5]$ is a weakly 3-ary $S$-primary hyperideal of $A$. 
\end{example}
\begin{example} \label{hanieh}
The construction $A=\mathbb{Z}_{12}/\mathbb{Z}_{12}^*$ is a Krasner hyperring where $\mathbb{Z}_{12}^*=\{1,5,7,11\}$ is a multiplicative
group of units  of $\mathbb{Z}_{12}=\{0,1,2,...,11\}$ by Example 3.2  \cite{rev2}. In the hyperring, the set $S=\{\mathbb{Z}_{12}, 3\mathbb{Z}_{12},9\mathbb{Z}_{12}\}$ is  3-ary multiplicative. The hyperideal $Q=\{0\mathbb{Z}_{12},4\mathbb{Z}_{12}\}$ with $\sqrt{Q}=\{0\mathbb{Z}_{12}^*,2\mathbb{Z}_{12}^*,4\mathbb{Z}_{12}^*, 6\mathbb{Z}_{12}^*\}$ is a weakly 3-ary $S$-primary hyperideal of $A$. 
\end{example}
Let $S \subseteq A$ is an $n$-ary multiplicative set. Recall from \cite{mah7} that a hyperideal $Q$ is called weakly $n$-ary $s$-prime if there exists $s \in S$ such that, whenever, $x_1^n \in A$, $0 \neq g(x_1^n) \in Q$ implies $g(s,x_i,1_A^{(n-2)}) \in Q$ for some $i \in \{1,\cdots,n\}$. Every weakly $n$-ary $S$-prime hyperideal is weakly $S$-primary. However, the converse is not true in general.
\begin{example}
In Example \ref{hanieh2}, The  hyperideal $Q=[0,0.5]$ is a weakly 3-ary $S$-primary hyperideal of $A$ where $S=\{1\}$. However, since $0 \neq g(0.6,0.7,0.8) \in Q$ but $0.6,0.7,0.8 \notin Q$, we get the result that  $Q$ is not weakly 3-ary $S$-prime hyperideal of $A$.
\end{example}
\begin{proposition}\label{21}
Let $Q$ and $P$ be  hyperideals of $A$ and $S \subseteq  A$ be an $n$-ary multiplicative set with  $Q \cap S=\varnothing$.  If  $Q$ is a weakly $n$-ary $S$-primary hyperideal of $A$ and  $P \cap S \neq \varnothing $, then $Q \cap P$ is a weakly $n$-ary $S$-primary hyperideal of $A$.
\end{proposition}
\begin{proof}
It is clear that  $(Q \cap P) \cap S=\varnothing$. Assume that  $a_1^n \in A$ such that  $0 \neq g(a_1^n) \in Q \cap P$. Since $Q$ is a weakly $n$-ary $S$-primary  hyperideal of $A$ and $0 \neq g(a_1^n) \in Q $, we conclude that there exists an $s \in S$ such that $g(s,a_i,1_A^{(n-2)}) \in Q$  or $g(a_1^{i-1},s,a_{i+1}^n) \in rad(Q)$ for some $i \in \{1,\cdots,n\}$. Now, take any $t \in P \cap S$. So $g(t^{(n-1)},s) \in S$. Then  $g(g(t^{(n-1)},s),a_i,1_A^{(n-2)})=g(t^{(n-1)},g(s,a_i,1_A^{(n-2)})) \in Q \cap P$ or $g(a_1^{i-1},g(t^{(n-1)},s),a_{i+1}^n)=g(g(t^{(n-1)}, g(a_1^{i-1},s,a_{i+1}^n)) \in rad(Q) \cap rad(P)=rad(Q \cap P)$. This means that  $Q \cap P$ is a weakly $n$-ary $S$-primary hyperideal of $A$.
\end{proof}
\begin{proposition} \label{22}
Let $Q_1^{n-1}$ and $Q$ be some hyperideals of $A$ and $S \subseteq  A$ be an $n$-ary multiplicative set such that $Q \cap S=\varnothing$. If $Q$ is a weakly $n$-ary $S$-primary hyperideal of $A$ and  $Q_j \cap S \neq \varnothing $ for each $j \in \{1,\cdots,n-1\}$, then $g(Q_1^{n-1},Q)$ is a weakly $n$-ary $S$-primary hyperideal of $A$.
\end{proposition}
\begin{proof}
Clearly, $g(Q_1^{n-1},Q) \cap S=\varnothing$.  Assume that $a_1^n \in A$ such that $0 \neq g(a_1^n) \in g(Q_1^{n-1},Q)$. Since $0 \neq g(a_1^n) \in Q $ and $Q$ is a weakly $n$-ary $S$-primary hyperideal of $A$, there exists an $s \in S$ such that $g(s,a_i,1_A^{(n-2)}) \in Q$ or $g(a_1^{i-1},s,a_{i+1}^n) \in rad(Q)$ for some $i \in \{1,\cdots,n\}$. Now, take  $t_j \in  Q_j \cap S$ for any $j \in \{1,\cdots,n-1\}$. Therefore we get  $g(t_1^{n-1},s) \in S$. If $g(s,a_i,1_A^{(n-2)}) \in Q$, then  $g(g(t_1^{n-1},s),a_i,1_A^{n-2})=g(t_1^{n-1},g(s,a_i,1_A^{(n-2)})) \in g(Q_1^{n-1},Q)$. If $g(a_1^{i-1},s,a_{i+1}^n) \in rad(Q)$, then there exists $u \in \mathbb{N}$ such that $g(g(a_1^{i-1},s,a_{i+1}^n)^{(u)},1_A^{(n-u)}) \in Q$ or $g_{(l)}(g(a_1^{i-1},s,a_{i+1}^n)^{(u)}) \in Q$ for $u=l(n-1)+1$. In the first case, we get  $g(g(a_1^{i-1},g(t_1^{n-1},s),a_{i+1}^n)^{(u)},1_A^{(n-u)})=g(g(t_1^{(u)},1_A^{(n-u)}),\cdots,g(t_{n-1}^{(u)},1_A^{(n-u)}),g(a_1^{i-1},s,a_{i+1}^n)^{(u)},1_A^{(n-u)})) \in g(Q_1^{n-1},Q)$ which implies $g(a_1^{i-1},g(t_1^{n-1},s),a_{i+1}^n) \in rad(g(Q_1^{n-1},Q))$.  Consequently,  $g(Q_1^{n-1},Q)$ is a weakly $n$-ary $S$-primary hyperideal of $A$. In the second case, a similar argument
completes the proof.
\end{proof}
Let $S \subseteq A$ be an $n$-ary multiplicative set satisfying  $S \cap Q=\varnothing$. Recall from \cite{mah6} that  a hyperideal of $Q$ of $A$ is called $n$-ary $S$-prime if there exists an $s \in S$ such that for every $x_1^n \in A$ with $ g(x_1^n) \in Q$,  we get $g(s,x_i,1_A^{(n-2)}) \in Q$ for some $i \in \{1,\cdots,n\}$.
\begin{example}
In Example \ref{hanieh}, the weakly 3-ary $S$-primary  hyperideal $Q$ is not a 3-ary $S$-prime hyperideal of $A$. Because, $g((2\mathbb{Z}_{12}^*)^{(2)},3\mathbb{Z}_{12}^*) \in Q$ but neither $g(2\mathbb{Z}_{12}^*,s,1) \in Q$ nor $g(3\mathbb{Z}_{12}^*,s,1) \in Q$ for all $s \in S$.
\end{example}
\begin{theorem} \label{bidi}
Let $S \subseteq A$ be an $n$-ary multiplicative set and  $\{0\}$ be an $n$-ary $S$-primary hyperideal of $A$. If $Q$ is a weakly $n$-ary $S$-primary hyperideal of $A$, then $rad(Q)$ is an $n$-ary $S$-prime hyperideal of $A$. 
\end{theorem}
\begin{proof}
Since $S \cap Q = \varnothing$,   we conclude that $S \cap rad(Q) =\varnothing$. Let $a_1^n \in A$ such that $g(a_1^n) \in rad(Q)$ and  $g(s,a_j,1^{(n-2)}) \notin rad(Q)$ for every $j \in \{1,\cdots,\widehat{i},\cdots,n\}$ and for all $s \in S$. Since $g(a_1^n) \in rad(Q)$, there exists $u \in \mathbb{N}$ such that $g(g(a_1^n)^{(u)},1_A^{(n-u)}) \in Q$ for $u \leq n$ or $g_{(l)}(g(a_1^n)^{(u)}) \in Q$ for $u=l(n-1)+1$. In the first possibilty, let  $0 \neq g(g(a_1^n)^{(u)},1_A^{(n-u)})=g(a_i^{(u)},g(g(a_1^{i-1},1_A,a_{i+1}^n)^{(u)},1_A^{(n-u)}),1_A^{(n-u-1)})$. Then there exists $s \in S$ such that $g(s,g(a_1^{i-1},1_A,a_{i+1}^n)^{(u)},1_A^{(n-u-1)}) \in Q$ or  $g(s,a_i^{(u)},1_A^{(n-u-1)}) \in rad(Q)$  as $Q$ is a weakly $n$-ary $S$-primary hyperideal of $A$.  If  the first case, we get\\ 

$\hspace{0.7cm}g(a_1^{(u)},g(a_2^{i-1},s,1_A,x_{i+1}^n)^{(u)},1^{(n-2u)}) \in Q$

$\Longrightarrow g(g(a_1^{(u)},1_A^{(n-u)}),g(a_2^{i-1},s,1_A,a_{i+1}^n)^{(u)},1_A^{(n-u-1)}) \in Q$

$\Longrightarrow g(s,g(a_1^{(u)},1_A^{(n-u)}),1_A^{(n-2)}) \in Q \ \text{or}\ g(s,g(a_2^{i-1},s,1_A,a_{i+1}^n)^{(u)},1_A^{(n-u-1)}) \in$

$\hspace{0.7cm} rad(Q)$

$\Longrightarrow g(g(s,a_1,1_A^{(n-2)})^{(u)},1_A^{(n-u)}) \in Q \ \text{or}\ g(g(a_2^{i-1},s^{(2)},a_{i+1}^n)^{(u)},1_A^{(n-u)}) \in$

$\hspace{0.7cm}rad(Q)$

$\Longrightarrow g(s,a_1,1_A^{(n-2)}) \in rad(Q) \ \text{or}\ g(g(a_2^{i-1},s^{(2)},a_{i+1}^n)^{(u)},1_A^{(n-u)}) \in rad(Q)$.\\

Since $g(s,a_1,1^{(n-2)}) \in rad(Q)$ is impossible, then \\

$ \hspace{0.7cm}g(g(a_2^{i-1},s^{(2)},a_{i+1}^n)^{(u)},1_A^{(n-t)}) \in rad(Q)$

$\Longrightarrow \exists v \in \mathbb{N}; \  g(g(g(a_2^{i-1},s^{(2)},a_{i+1}^n)^{(u)},1_A^{(n-u)})^{(v)}, 1_A^{(n-v)}) \in Q$

$\Longrightarrow g(g(a_2^{(u+v)},1_A^{(n-u-v)}),g(g(a_3^{i-1},s^{(2)},1_A,a_{i+1}^n)^{(u)},1_A^{(n-v)})^{(v)},1_A^{(n-v-1)}) \in Q$

$\Longrightarrow g(s,a_2,1^{(n-2)}) \in rad(Q) \ \text{or} \ g(g(g(a_3^{i-1},s^{(3)},a_{i+1}^n)^{(u)},1_A^{(n-u)})^{(v)},1_A^{(n-v)}) \in $

$\hspace{0.7cm}rad(Q)$

$\vdots$

$\Longrightarrow \cdots \ \text{or} \ g(s,a_n,1^{(n-2)}) \in rad(Q)$\\
which is contradiction with the fact that $g(s,x_j,1^{(n-2)}) \notin rad(Q)$ for all $j \in \{1,\cdots,\widehat{i},\cdots,n\}$. In the second case,  we get $g(g(s,a_i,1_A^{(n-2)})^{(u)},1_A^{(n-u)}) \in Q$ which implies $g(s,a_i,1_A^{(n-2)}) \in rad(Q)$, as needed. Now, let $0= g(g(a_1^n)^{(u)},1_A^{(n-u)})=g(a_i^{(u)},g(g(a_1^{i-1},1_A,a_{i+1}^n)^{(u)},1_A^{(n-u)}),1_A^{(n-u-1)})$. Then there exists $t \in S$ such that $g(s,g(a_1^{i-1},1_A,a_{i+1}^n)^{(u)},1_Q^{(n-u-1)})=0$ or  $g(s,a_i^{(u)},1_A^{(n-u-1)}) \in rad(0)$  as $\{0\}$ is an $n$-ary $S$-primary hyperideal of $A$. In the first case, we conclude that \\

$\hspace{0.7cm}g(a_1^{(u)},g(a_2^{i-1},s,1_A,x_{i+1}^n)^{(u)},1^{(n-2u)}) = 0$

$\Longrightarrow g(g(a_1^{(u)},1_A^{(n-u)}),g(a_2^{i-1},s,1_A,a_{i+1}^n)^{(u)},1_A^{(n-u-1)}) =0$

$\Longrightarrow g(s,g(a_1^{(u)},1_A^{(n-u)}),1_A^{(n-2)}) =0 \ \text{or}\ g(s,g(a_2^{i-1},s,1_A,a_{i+1}^n)^{(u)},1_A^{(n-u-1)}) \in$

$\hspace{0.7cm} rad(0)$

$\Longrightarrow g(g(s,a_1,1_A^{(n-2)})^{(u)},1_A^{(n-u)}) =0 \ \text{or}\ g(g(a_2^{i-1},s^{(2)},a_{i+1}^n)^{(u)},1_A^{(n-u)}) \in$

$\hspace{0.7cm}rad(0)$

$\Longrightarrow g(s,a_1,1_A^{(n-2)}) \in rad(0) \ \text{or}\ g(g(a_2^{i-1},s^{(2)},a_{i+1}^n)^{(u)},1_A^{(n-u)}) \in rad(0)$.\\

Since $g(s,a_1,1^{(n-2)}) \notin rad(Q)$, we get $g(s,a_1,1^{(n-2)}) \notin rad(0)$. Hence we have \\

$ \hspace{0.7cm}g(g(a_2^{i-1},s^{(2)},a_{i+1}^n)^{(u)},1_A^{(n-t)}) \in rad(0)$

$\Longrightarrow \exists v \in \mathbb{N}; \  g(g(g(a_2^{i-1},s^{(2)},a_{i+1}^n)^{(u)},1_A^{(n-u)})^{(v)}, 1_A^{(n-v)}) =0$

$\Longrightarrow g(g(a_2^{(u+v)},1_A^{(n-u-v)}),g(g(a_3^{i-1},s^{(2)},1_A,a_{i+1}^n)^{(u)},1_A^{(n-v)})^{(v)},1_A^{(n-v-1)}) =0$

$\Longrightarrow g(s,a_2,1^{(n-2)}) \in rad(Q) \ \text{or} \ g(g(g(a_3^{i-1},s^{(3)},a_{i+1}^n)^{(u)},1_A^{(n-u)})^{(v)},1_A^{(n-v)}) \in $

$\hspace{0.7cm}rad(0)$

$\vdots$

$\Longrightarrow \cdots \ \text{or} \ g(s,a_n,1^{(n-2)}) \in rad(0)$

which is impossible. In the second case,  we get $g(g(s,a_i,1_A^{(n-2)})^{(u)},1_A^{(n-u)}) \in rad(0)$ which implies $g(s,a_i,1_A^{(n-2)}) \in rad(Q)$.  Consequently, $rad(Q)$ is an $n$-ary $S$-prime hyperideal of $A$. In the second possibility,  by using a similar argument, one can easily get the result that $rad(Q)$ is an $n$-ary $S$-prime hyperideal of $A$.
\end{proof}
\begin{proposition} \label{23}
Let $Q_1^n$ be some hyperideals of $A$ and $S \subseteq  A$ be an $n$-ary multiplicative set such that $Q_j \cap S=\varnothing$ for each $j \in \{1,\cdots,n\}$. If $Q_j$ is a weakly $n$-ary $S$-primary hyperideal of $A$ for any $j \in \{1,\cdots,n\}$ such that $rad(Q_j)=rad(Q_k)$ for every $j,k \in \{1,\cdots,n\}$, then $\cap_{j=1}^n Q_j$ is a weakly $n$-ary $S$-primary hyperideal of $A$.
\end{proposition} 
\begin{proof}
Let $Q_j$ be a weakly $n$-ary $S$-primary hyperideal of $A$ for each $j \in \{1,\cdots,n\}$. Then there exists $t_j \in S$ such that for all $x_1^n \in A$ with $0 \neq g(x_1^n) \in Q_j$, either $g(t_j,x_i,1_A^{(n-2)}) \in Q_j$ or $g(x_1^n,t_j,x_{i+1}^n) \in rad(Q_j)$ for some $i \in \{1,\cdots,n\}$. Put $s=g(t_1^n)$. So $s \in S$. Now, let $a_1^n \in A$ such that $0 \neq g(a_1^n) \in \cap_{j=1}^n Q_j$ and $g(s,a_i,1_A^{(n-2)}) \notin \cap_{j=1}^n Q_j$ for all $i \in \{1,\cdots,n\}$. Then there exists $h \in \{1,\cdots,n\}$ with $g(s,a_i,1_A^{(n-2)}) \notin Q_h$ which means $g(t_h,a_i,1_A^{(n-2)}) \notin Q_h$. Since $0 \neq g(a_1^n) \in Q_h$, we conclude that $g(a_1^{i-1},t_h,a_{i+1}^n) \in rad(Q_h)$ and so $g(a_1^{i-1},s,a_{i+1}^n)=g(a_1^{i-1},g(t_1^n),a_{i+1}^n)=g(t_1^{h-1},g(a_1^{i-1},t_h,a_{i+1}^n),t_{h+1}^n) \in rad(Q_h)$. Since $rad(Q_j)=rad(Q_k)$ for each $j,k \in \{1,\cdots,n\}$,  we get the result that  $g(a_1^{i-1},s,a_{i+1}^n) \in rad(Q_h)=\cap_{j=1}^n rad(Q_j)=rad(\cap_{j=1}^n Q_j)$. Consequently, $\cap_{j=1}^n Q_j$ is a weakly $n$-ary $S$-primary hyperideal of $A$.
\end{proof}
\begin{proposition} \label{2pezeshk}
Assume that  $S \subseteq T \subseteq A$ are two $n$-ary multiplicative sets such that for each $s \in T$, there exists $s^\prime \in T$ with $g(s^{(n-1)},s^{\prime}) \in S$. If $Q$ is a weakly  $n$-ary $T$-primary hyperideal of $A$, then $Q$ is a weakly $n$-ary $S$-primary hyperideal of $A$.
\end{proposition}
\begin{proof}
Assume that $a_1^n \in A$ such that $0 \neq g(a_1^n) \in Q$. Since $Q$ is a weakly $n$-ary $T$-primary  hyperideal,  there exists $s \in T$ such that $g(s,a_i,1^{(n-2)}) \in Q$ or $g(a_1^{i-1},s,a_{i+1}^n) \in rad(Q)$ for some $i \in \{1,\cdots,n\}$. By the assumption, there exists $s^{\prime} \in T$ such that $g(s^{(n-1)},s^\prime) \in S$. Put $t=g(s^{(n-1)},s^\prime)$. Let  $g(s,a_i,1^{(n-2)}) \in Q$. Then we get the result that  $g(t,a_i,1_A^{(n-2)})=g(g(s^{(n-1)},s^{\prime}),a_i,1_A^{(n-2)})
=g(g(s^{(n-2)},s^{\prime},1_A),g(s,a_i,1^{(n-2)}),1_A^{(n-2)}) \in Q$. Now, assume that   $g(a_1^{i-1},s,a_{i+1}^n) \in rad(Q)$. Then we conclude that $g(a_1^{i-1},t,a_{i+1}^n)=g(a_1^{i-1},g(s^{(n-1)},s^{\prime}),a_{i+1}^n)
=g(g(s^{n-2},s^{\prime},1_A),g(a_1^{i-1},s,a_{i+1}^n),1_A^{(n-2)}) \in rad(Q)$. Thus $Q$ is a weakly $n$-ary $S$-primary hyperideal of $A$.
\end{proof}

\begin{theorem} \label{24}
Let  $S \subseteq A$ be an $n$-ary multiplicative set and $Q$ be a hyperideal of $A$ such that $1_A \in S$ and $S \cap Q=\varnothing$. Then $Q$ is a weakly $n$-ary $S$-primary  hyperideal of $A$ if and only if $Q$ is a weakly $n$-ary $S^{\prime}$-primary  hyperideal where $S^{\prime}=\{x \in A \ \vert \ \frac{x}{1_A} \ \text{is invertible in} \ S^{-1}A \}$.
\end{theorem}
\begin{proof}
($\Longrightarrow$) Since $S^{\prime}$ is an $n$-ary multiplicative subset of $A$ containing $S$ and $Q$ is a weakly $n$-ary $S$-primary hyperideal of $A$, we are done.\\ 

($\Longleftarrow$) Suppose that $s \in S^{\prime}$. This implies that $\frac{s}{1_A}$ is invertible in $S^{-1}A$. Then there exists $x \in A$ and $t \in S$ such that $G(\frac{s}{1_A},\frac{x}{t},\frac{1_A}{1_A}^{(n-2)})=\frac{g(s,x,1_A^{(n-2)})}{g(t,1_A^{(n-1)})}=\frac{1_A}{1_A}$.  So there exists $t^{\prime} \in S$ with $0 \in g(t^{\prime}, f(g(s,x,1_A^{(n-2)}),-g(t,1_A^{(n-1)}),0^{(m-2)}),1_A^{(n-2)})=f(g(t^{\prime}, s,x,1_A^{(n-3)}),-g(t^{\prime},t,1_A^{(n-2)}),0^{(m-2)})$. 
Since  $g(t^{\prime},t,1_A^{(n-2)}) \in S$, we get the result that $g(t^{\prime}, s,x,1_A^{(n-3)}) \in f(g(t^{\prime},t,1_A^{(n-2)}),0^{(m-1)}) \subseteq S$.
Let  $t^{\prime \prime }=g(t^{\prime},x,1_A^{(n-2)})$. Since $G(\frac{g(t^{\prime},x,1_A^{(n-2)})}{1_A},\frac{g(s,1_A^{(n-1)})}{g(t^{\prime},x,s,1_A^{(n-3)})},\frac{1_A}{1_A}^{(n-2)})=\frac{g(t^{\prime},x,s,1_A^{(n-3)})}{g(t^{\prime},x,s,1_A^{(n-3)})}=\frac{1_A}{1_A}$, we get $t^{\prime \prime } \in S^{\prime}$.  Therefore we obtain $g(s^{(n-1)},g({{t^{\prime \prime}}^{(n-1)}},g(s,t^{\prime \prime},1_A^{(n-2)})))=g(g(s,t^{\prime \prime},1_A^{(n-2)})^n) \in S$. Put  $s^{\prime}=g({{t^{\prime \prime}}^{(n-1)}},g(s,t^{\prime \prime},1_A^{(n-2)}))$. So $s^{\prime} \in S^{\prime}$. Since $g(s^{(n-1)},s^{\prime}) \in S$,  $Q$ is a weakly $n$-ary $S$-primary hyperideal of $A$ by Proposition \ref{2pezeshk}.
\end{proof}
\begin{proposition} \label{25}
Let $Q$ be   a hyperideal of $A$ and $S \subseteq  A$ be an $n$-ary multiplicative set  with $Q \cap S=\varnothing$. If $(Q : s)$  is a weakly $n$-ary primary hyperideal of $A$  for some $s \in S$, then $Q$ is a weakly $n$-ary $S$-primary hyperideal of $A$.
\end{proposition}
\begin{proof}
 Assume that  $(Q : s)$ is  a weakly $n$-ary primary hyperideal of $A$  for some $s \in S$.  Let  $a_1^n \in A$ such that $0 \neq g(a_1^n)\in Q$. Therefore we have $0 \neq g(a_1^n)\in (Q : s)$ because $Q \subseteq (Q : s)$. Since the hyperideal $(Q : s)$ is  weakly $n$-ary primary, we have $a_i \in (Q : s)$ or $g(a_1^{i-1},1_A,a_{i+1}^n) \in rad((Q : s))$ for some $i \in \{1,\cdots,n\}$.  In the first possibility, we have $g(s,a_i,1_A^{(n-2)})\in Q$ and so   $Q$ is a weakly $n$-ary $S$-prime hyperideal of $A$. In the second possibility, there exists $u \in \mathbb{N}$ such that $g(g(a_1^{i-1},1_A,a_{i+1}^n)^{(u)},1_A^{(n-u)}) \in (Q : s) $ or $g_{(l)}(g(a_1^{i-1},1_A,a_{i+1}^n)^{(u)}) \in (Q : s)$ for $u=l(n-1)+1$. In the first case,  we have $g(g(a_1^{i-1},s,a_{i+1}^n)^{(u)},1_A^{(n-u)})=g(g(s^{(u-1)},1_A^{(n-u+1)}),g(s,g(a_1^{i-1},1_A,a_{i+1}^n)^{(u)},1_A^{(n-u-1)}),1_A^{(n-2)})) \in Q$  which implies $g(a_1^{i-1},s,a_{i+1}^n) \in rad(Q)$ and so we are done. In the second case, one can complete the proof similarly. 
\end{proof}
The notion of Krasner $(m,n)$-hyperring of fractions was  studied in \cite{mah5}.
\begin{theorem}
Let $S_1 \subseteq S_2 \subseteq A$ be two $n$-ary multiplicative set with $1_A \in S_1$ and $Q$ be a hyperideal of $A$ such that $S_2 \cap Q=\varnothing$. If $Q$ is a weakly $n$-ary $S_1$-primary hyperideal of $A$, then
\begin{itemize} 
\item[\rm{(i)}]~  $S_2^{-1}Q$ is a weakly $n$-ary $S_2^{-1}S_1$-primary hyperideal of $S_2^{-1}A$ where $S_2^{-1}S_1=\{\frac{s}{t} \ \vert \  s \in S_1, t \in S_2\}$. 
\item[\rm{(ii)}]~$S_2^{-1}Q \cap A=(Q : s) \cup 0_{S_2}$ for some $s \in S_1$ such that  $0_{S_2}=\{x \in A \ \vert \ g(x,t,1_A^{(n-2)})=0 \ \text{for some} \ t \in S_2 \}$.
\end{itemize}
\end{theorem}
\begin{proof}
(i) It is easy to see that $S_2^{-1}S_1$ is an $n$-ary multiplicative subset of $S_2^{-1}A$. Let $Q$ be a weakly $n$-ary $S_1$-primary hyperideal of $A$ and $S_2^{-1}S_1 \cap S_2^{-1}Q \neq \varnothing$. Assume that $\frac{x}{t} \in S_2^{-1}S_1 \cap S_2^{-1}Q$ This means that $x \in S_1$ and $g(t^{\prime},x,1_A^{(n-2)}) \in Q$ for some $t^{\prime} \in S_2$. Then we have $x \in S_2$ and so $g(t^{\prime},x,1_A^{(n-2)}) \in S_2 \cap Q$ which is impossible. Then $S_2^{-1}S_1 \cap S_2^{-1}Q = \varnothing$. Since $Q$ is a weakly $n$-ary $S_1$-primary hyperideal of $A$, there exists $s \in S_1$ such that, whenever $x_1^n \in A$, $0 \neq g(x_1^n) \in Q$ implies $g(s,x_i,1_A^{(n-2)}) \in Q$ or $g(x_1^{i-1},s,x_{i+1}^n) \in rad(Q)$. Let $\frac{x_1}{t_1},\cdots,\frac{x_n}{t_n} \in S_2^{-1}A$ such that $0 \neq G(\frac{x_1}{t_1},\cdots,\frac{x_n}{t_n}) \in S_2^{-1}Q$ and $G(\frac{x_1}{t_1},\cdots,\frac{x_{i-1}}{t_{i-1}},\frac{s}{1_A},\frac{x_{i+1}}{t_{i+1}},\cdots,\frac{x_n}{t_n}) \notin rad(S_2^{-1}Q)$ for all $i \in \{1,\cdots,n\}$. This means $0 \neq \frac{g(x_1^n)}{g(t_1^n)} \in S_2^{-1}Q$ and so $0 \neq g(t,g(x_1^n),1_A^{(n-2)}) \in Q$ for some $t \in S_2$ and $g(x_1^{i-1},s,x_{i+1}^n) \notin rad(Q)$. This implies that  $g(s,t,x_i,1_A^{(n-3)})=g(s,g(t,x_i,1_A^{(n-2)}) \in Q$ as $Q$ is a weakly $n$-ary $S_1$-primary hyperideal of $A$. Then we get the result that $G(\frac{s}{1_A},\frac{x_i}{t_i},\frac{1_A}{1_A}^{(n-2)})=G(\frac{s}{1_A},\frac{x_i}{t_i},\frac{t}{t},\frac{1_A}{1_A}^{(n-3)})=\frac{g(s,x_i,t,1_A^{(n-3)})}{g(t_i,t,1_A^{(n-2)})} \in S_2^{-1}Q$. Consequenly, $S_2^{-1}Q$ is a weakly $n$-ary $S_2^{-1}S_1$-primary hyperideal of $S_2^{-1}A$.\\

(ii) Let  $Q$ be a weakly $n$-ary $S_1$-primary hyperideal of $A$. Then there exists $s \in S_1$ such that, whenever $x_1^n \in A$, $0 \neq g(x_1^n) \in Q$ implies $g(s,x_i,1_A^{(n-2)}) \in Q$ or $g(x_1^{i-1},s,x_{i+1}^n) \in rad(Q)$. It is clear that $(Q : s) \cup 0_{S_2} \subseteq S_2^{-1}Q \cap A$. Now, take any $a \in S_2^{-1}Q \cap A$. Since $\frac{a}{1_A}=\frac{x}{u}$ for some $x \in Q$ and $u \in S_2$, we obtain $g(t,a,1_A^{(n-2)}) \in Q$ for some $t \in S_2$. Let $g(t,a,1_A^{(n-2)})=0$. Then we conclude that $a \in 0_{S_2}$. Now, let $g(t,a,1_A^{(n-2)}) \neq 0$. Then we get the result that $g(s,t,1_A^{(n-2)}) \in rad(Q)$ which is impossible as $S_2 \cap rad(Q)=\varnothing$ or we have $g(s,a,1_A^{(n-2)}) \in Q$ which implies $a \in (Q : s)$ and so $S_2^{-1}Q \cap A \subseteq (Q : s) \cup 0_{S_2}$. Thus we have $S_2^{-1}Q \cap A=(Q : s) \cup 0_{S_2}$.
\end{proof}

Let $(A_1, f_1, g_1)$ and $(A_2, f_2, g_2)$ be two commutative Krasner $(m, n)$-hyperrings. Recall from \cite{d1} that a  mapping
$\psi : A_1 \longrightarrow A_2$ is called a homomorphism if for all $x^m _1, y^n_ 1 \in A_1$ we have
\begin{itemize}
\item[\rm{(i)}]~$\psi(f_1(x_1,\cdots, x_m)) = f_2(\psi(x_1),\cdots, \psi(x_m)),$
\item[\rm{(ii)}]~$\psi(g_1(y_1,\cdots, y_n)) = g_2(\psi(y_1),\cdots,\psi(y_n))$
\item[\rm{(iii)}]~$\psi(1_{A_1})=1_{A_2}.$
\end{itemize}
\begin{theorem} \label{26}
Assume that $\psi:A_1 \longrightarrow A_2$ is a monomorphism where $(A_1,f_1,g_1)$ and $ (A_2,f_2,g_2)$ are  two commutative Krasner $(m,n)$-hyperrings  and $S \subseteq A_1$  is an $n$-ary multiplicative set such that the element $0$ is not in $\psi(S)$. If  $Q_2$ is a weakly  $n$-ary $\psi(S)$-primary hyperideal of $A_2$, then $\psi^{-1}(Q_2)$ is a weakly $n$-ary $S$-primary hyperideal of $A_1$. 
\end{theorem}
\begin{proof}
 Let  $Q_2$ be a weakly $n$-ary $\psi(S)$-primary hyperideal of $A_2$. Therefore there exists $s \in S$ such that for all $y_1^n \in A_2$ with $0 \neq g_2(y_1^n) \in Q_2$,  $g_2(\psi(s),y_i,1_{A_2}^{(n-2)}) \in Q_2$ or $g_2(y_1^{i-1},\psi(s),y_{i+1}^n) \in rad(Q_2)$ for some $i \in \{1,\cdots,n\}$. Put $Q_1=\psi^{-1}(Q_2)$. It is clear that $Q_1 \cap S = \varnothing$. Assume that $0 \neq g_1(x_1^n) \in Q_1$ for $x_1^n \in A_1$. Then $0 \neq \psi(g_1(x_1^n))=g_2(\psi(x_1),...,\psi(x_n)) \in Q_2$ as $\psi$ is a monomorphism. Therefore we get the result that  $g_2(\psi(s),\psi(x_i),1_{A_2}^{(n-2)})=\psi(g_1(s,x_i,1_{A_1}^{(n-2)})) \in Q_2$  or $g_2(\psi(x_1),\cdots,\psi(x_{i-1}),\psi(s),\psi(x_{i+1}),\cdots,\psi(x_n))=\psi(g_1(x_1^{i-1},s,x_{i+1}^n)) \in rad(Q_2)$. In the first case, we get  $g_1(s,x_i,1_{A_1}^{(n-2)}) \in \psi^{-1}(Q_2)=Q_1$. In the second case, we obtain $g_1(x_1^{i-1},s,x_{i+1}^n) \in \psi^{-1}(rad(Q_2))=rad(\psi^{-1}(Q_2))=rad(Q_1)$. Thus $Q_1=\psi^{-1}(Q_2)$ is a weakly $n$-ary $S$-primary hyperideal of $A_1$.
\end{proof}
\begin{corollary}
Assume that $A_1$ is a subhyperring of $A_2$ and $S \subseteq A_1$  is an $n$-ary multiplicative set. If $Q_2$ is a weakly $n$-ary $S$-primary  hyperideal of $A_2,$ then $A_1 \cap Q_2$ is a weakly $n$-ary $S$-primary hyperideal of $A_1$.
\end{corollary}
\begin{proof}
Consider the monomorphism $\psi : A_1 \longrightarrow A_2$ , defined by $\psi(x)=x$. Since $\psi^{-1}(Q_2)=A_1 \cap Q_2$, we get the result  that $A_1 \cap Q_2$ is a weakly $n$-ary $S$-primary hyperideal of $A_1$, by Theorem \ref{26}.
\end{proof}
Let  $(A_1, f_1, g_1)$ and $(A_2, f_2, g_2)$ be two commutative Krasner $(m,n)$-hyperrings and, $1_{A_1}$ and $1_{A_2}$ be  scalar identities of $A_1$ and $A_2$, respectively. Then the triple $(A_1 \times A_2, f _1 \times f_2 ,g_1 \times g_2 )$ is a Krasner $(m, n)$-hyperring  where $m$-ary hyperoperation
$f _1 \times f_2 $ and $n$-ary operation $g_1 \times g_2$ are defined as follows:

$\hspace{1cm} f_1 \times f_2((x_{1}, y_{1}),\cdots,(x_m,y_m)) = \{(a,b) \ \vert \ \ a \in f_1(x_1^m), b \in f_2(y_1^m) \}$

$\hspace{1cm} g_1 \times g_2 ((a_1,b_1),\cdots,(a_n,b_n)) =(g_1(a_1^n),g_2(b_1^n)) $,\\
for  $x_1^m,a_1^n \in A_1$ and $y_1^m,b_1^n \in A_2$ \cite{mah2}. 
\begin{theorem} \label{cart1}
Suppose that  $Q_1$ and $Q_2$ are  nonzero hyperideals of $A_1$ and $A_2$, respectively,  where $(A_1, f_1, g_1)$ and $(A_2, f_2, g_2)$ are two commutative Krasner $(m,n)$-hyperrings and   $S_1 \subseteq A_1$ and $S_2 \subseteq A_2$ are two $n$-ary multiplicative sets. Then the following statements are equivalent:
 \begin{itemize} 
\item[\rm{(i)}]~ $Q=Q_1 \times Q_2$ is a weakly $n$-ary $S_1 \times S_2$-primary hyperideal of $A_1 \times A_2$.
\item[\rm{(ii)}]~ $Q_1$ is an $n$-ary $S_1$-primary hyperideal of $A_1$  and $S_2 \cap Q_2 \neq \varnothing$ or $Q_2$ is an $n$-ary $S_2$-primary hyperideal of $A_2$  and $S_1 \cap Q_1  \neq \varnothing$
\item[\rm{(iii)}]~ $Q=Q_1 \times Q_2$ is an $n$-ary $S_1 \times S_2$-primary hyperideal of $A_1 \times A_2$.
\end{itemize} 
\end{theorem}
\begin{proof}
(i) $\Longrightarrow$ (ii) Suppose that $(0,0) \neq (x,y) \in Q$ for some $x \in Q_1$ and $y \in Q_2$. Then  we get $(0,0) \neq (x,y)=g_1 \times g_2((x,1_{A_2}),(1_{A_1},1_{A_2})^{(n-2)},(1_{A_1},y))\in Q$. Then there exists $(s_1,s_2) \in S_1 \times S_2$ such that $g_1 \times g_2((s_1,s_2),(x,1_{A_2}),(1_{A_1},1_{A_2})^{(n-2)})=(g_1(s_1,x,1_{A_1}^{(n-2)}),g_2(s_2,1_{A_2}^{(n-1)})) \in Q$ or $g_1 \times g_2 ((s_1,s_2),(1_{A_1},y),(1_{A_1},1_{A_2})^{(n-2)})=(g_1(s_1,1_{A_1}^{(n-1)}),
g_2(s_2,y,1_{A_2}^{(n-2)})) \in rad(Q)=rad(Q_1) \times rad(Q_2)$ as $Q$ is a weakly $n$-ary $S_1 \times S_2$-primary hyperideal of $A_1 \times A_2$. In the first case, we get $S_2 \cap Q_2 \neq \varnothing$. In the second case, there exists $u \in \mathbb{N}$ such that $g_1(s_1^{(u)},1_{A_1}^{(n-u)}) \in Q_1$ for $u \leq n$ or $g_{1_{(l)}}(s_1^{(u)}) \in Q_1$ for $u=l(n-1)+1$. This means  $S_1 \cap Q_1  \neq \varnothing$. From $S_2 \cap Q_2 \neq \varnothing$ it follows that $S_1 \cap Q_1 = \varnothing$ as $Q \cap (S_1 \times S_2)= \varnothing$. Then there exists $0 \neq a \in S_2 \cap Q_2$. Now, let $x_1^n \in A_1$ such that $g(x_1^n) \in Q_1$. Then we have $0 \neq (g_1(x_1^n),g_2(a,1_{A_2}^{(n-1)}))=g_1 \times g_2 ((x_1,a),(x_2,1_{A_2}),\cdots(x_n,1_{A_2})) \in Q$. Therefore we get  $(g_1(s_1,x_1,1_{A_1}^{(n-2)}),g_2(s_2,a,1_{A_2}^{(n-2)}))=g_1 \times g_2((s_1,s_2),(x_1,a),(1_{A_1},1_{A_2})^{(n-2)}) \in Q$ or   $(g_1(s_1,x_i,1_{A_1}^{(n-2)}),g_2(s_2,1_{A_2}^{(n-1)}))=g_1 \times g_2((s_1,s_2),(x_i,1_{A_2}),(1_{A_1},1_{A_2})^{(n-2)})\in Q$ or $(g_1(s_1,x_2^n),g_2(s_2,1_{A_2}))=g_1 \times g_2((s_1,s_2),(x_2,1_{A_2}),\cdots,(x_n,1_{A_2})) \in rad(Q)=rad(Q_1) \times rad(Q_2)$ or $(g_1(x_1^{i-1},s_1,x_{i+1}^n),g_2(a,1_{A_2}^{(n-1)})) \in rad(Q)=rad(Q_1) \times rad(Q_2)$ for some $ i \in \{2,\cdots,n\}$. Hence  $g_1(s_1,x_i,1_{A_1}) \in Q_1$ or $g_1(x_1^{i-1},s_1,x_{i+1}^n) \in rad(Q_1)$ which implies
$Q_1$ is an $n$-ary $S_1$-primary hyperideal of $A_1$.\\

(ii) $\Longrightarrow$ (iii) Assume that $Q_1$ is an $n$-ary $S_1$-primary hyperideal of $A_1$  and $S_2 \cap Q_2  \neq \varnothing$. Then there exists $s_2 \in S_2 \cap Q_2$. Let $x_1^n \in A_1$ and $y_1^n \in A_2$ such that $g((x_1,y_1),\cdots,(x_n,y_n))=(g_1(x_1^n),g_2(y_1^n)) \in Q$. Since $g_1(x_1^n) \in Q_1$ and $Q_1$ is an $n$-ary $S_1$-primary hyperideal of $A_1$, there exists $s_1 \in S_1$ satisfying  $g_1(s_1,x_i,1_{A_1}^{(n-2)}) \in Q_1$ or $g_1(x_1^{i-1},s_1,x_{i+1}^n) \in rad(Q_1)$ for some $i \in \{1,\cdots,n\}$. Therefore we have $(g_1(s_1,x_i,1_{A_2}^{(n-2)}),g_2(s_2,y_i,1_{A_1}^{(n-2)})) \in Q$ or $(g_1(x_1^{i-1},s_1,x_{i+1}^n),g_2(y_1^{i-1},s_2,y_{i+1}^n)) \in rad(Q_1) \times Q_2 $. Then we conclude that   $g_1 \times g_2((s_1,s_2),(x_i,y_i),(1_{A_1},1_{A_2})^{(n-2)}) \in Q$ or $g_1 \times g_2((x_1,y_1),\cdots,(x_{i-1},y_{i-1}),(s_1,s_2),(x_{i+1},y_{i+1}),\cdots,(x_n,y_n)) \in  rad(Q)$. Consequently,  $Q$ is an $n$-ary $S_1 \times S_2$-primary hyperideal of $A_1 \times A_2$. Similarly, one can show that $Q$ is an $n$-ary $S_1 \times S_2$-primary hyperideal of $A_1 \times A_2$ where $Q_2$ is an $n$-ary $S_2$-primary hyperideal of $A_2$  and $S_1 \cap Q_1 \neq \varnothing$.\\

(iii) $\Longrightarrow$ (i) It is clear.
\end{proof}
Now, the following result concluded by using mathematical induction on $k$ and   previous theorem is given.
\begin{theorem} \label{cart2}
Suppose that  $Q_1^k$ are  nonzero hyperideals of $A_1^k$, respectively,  where $(A_k, f_k, g_k), \cdots, (A_k, f_k, g_k)$ are  commutative Krasner $(m,n)$-hyperrings and   $S_1 \subseteq A_1,\cdots, S_k \subseteq A_k$ are  $n$-ary multiplicative sets. Then  $Q=Q_1 \times \cdots \times Q_k$ is a weakly $n$-ary $(S_1 \times \cdots \times S_k)$-primary hyperideal of $A_1 \times \cdots \times A_k$ if and only if  $Q_j$ is an $n$-ary $S_j$-primary hyperideal of $A_j$ for some $j \in \{1,\cdots,k\}$ and $S_h\cap Q_h \neq \varnothing$ for all $h \neq j$.
\end{theorem}
Recall from \cite{davvazz} that a hyperideal $Q$ of $A$  is  strongly weakly $n$-ary primary  if $0 \neq g(Q_1^n) \subseteq Q$ for each hyperideals $Q_1^n$ of $A$  implies  $Q_i \subseteq Q$ or $g(Q_1^{i-1},1_A,Q_{i+1}^n) \subseteq rad(Q)$ for some $i \in \{1,\cdots,n\}$. Suppose that  $S \subseteq A$ is an $n$-ary multiplicative set  such that  $S \cap Q= \varnothing$.  A hyperideal $Q$ of $A$  is called  strongly weakly $n$-ary $S$-primary  if there exists an element $s \in S$ such that for all hyperideal  $Q_1^n $ in $ A$ if $0 \neq g(Q_1^n) \subseteq Q$, we get $g(s,Q_i,1_A^{(n-2)}) \subseteq Q$ or $g(Q_1^{i-1},s,Q_{i+1}^n) \subseteq rad(Q)$ for some $i \in \{1,\cdots,n\}$.  In this case, $s$ is  said to be a strongly weakly $S$-primary element of $Q$. It is clear that every strongly weakly $n$-ary $S$-primary hyperideal of $A$  is a weakly $n$-ary $S$-primary  hyperideal.
\begin{theorem} \label{27}
Let $Q$ be a strongly weakly $n$-ary $S$-primary hyperideal of $A$ such that  $s$  is a strongly weakly $S$-primary element of $Q$. If $g(x_1^n)=0$ for $x_1^n \in A$ but $g(s,x_i,1_A^{(n-2)}) \notin Q$ and $g(x_1^{i-1},s,x_{i+1}^n) \notin rad(Q)$ for all $i \in \{1,\cdots,n\}$,    then $g(x_1,\cdots,\widehat{x_{j_1}}, \cdots,\widehat{x_{j_2}},\cdots, \widehat{x_{j_u}}, \cdots, Q^{(u)})=0$ for each $j_1,\cdots,j_u \in \{1,\cdots,n\}$. 
\end{theorem}
\begin{proof}
We use the induction on $u$. Let $u=1$. Suppose on the contrary that  $g(x_1^{i-1},Q,x_{i+1}^n) \neq 0$ for some $i \in \{1,\cdots,n\}$. Hence  we have $0 \neq g(x_1^{i-1},x,x_{i+1}^n) \in Q$ for some $x \in Q$. Then   $0 \neq g(x_1^{i-1},x,x_{i+1}^n)=f(g(x_1^n),g(x_1^{i-1},x,x_{i+1}^n),0^{(m-2)})=g(x_1^{i-1},f(x,x_i,0^{(m-2)}),x_{i+1}^n) \subseteq Q$. By the hypothesis, we get the result that $g(s,f(x,x_i,0^{(m-2)}),1_A^{(n-2)})=f(g(s,x,1_A^{(n-2)}),g(s,x_i,1_A^{(n-2)}),0^{(m-2)}) \subseteq Q$ which implies  $g(s,x_i,1_A^{(n-2)}) \in Q$ or $g(x_1^{i-1},s,x_{i+1}^n) \in rad(Q)$ for some $i \in \{1,\cdots,n\}$ which is impossible. Now, assume that the
claim is true for all positive integers being  less than $u$. Let $g(x_1,\cdots,\widehat{x_{j_1}}, \cdots,\widehat{x_{j_2}},\cdots, \widehat{x_{j_u}}, \cdots, Q^{(u)}) \neq 0$ for some $j_1,\cdots,j_u \in \{1,\cdots,n\}$. . Without loss of generality, we eliminate $x_1^u$. So we have $g(x_{u+1},\cdots,x_n,Q^{(u)}) \neq 0$. Then   $0 \neq g(x_{u+1},\cdots,x_n,a_1^u) \in Q$ for some $a_1^u \in Q$. By induction hypothesis, we get $0 \neq g(f(x_1,a_1,0^{(m-2)}),\cdots,f(x_u,a_u,0^{(m-2)}),x_{u+1}^n) \subseteq Q$. By the hypothesis, we get the result that  $g(s,f(x_i,a_i,0^{(m-2)}),1_A^{(n-2)}) \in Q$  or $g(f(x_1,a_1),\cdots,\widehat{f(x_i,a_i,0^{(m-2)})},\cdots,f(x_u,a_u,0^{(m-2)}),s,x_{u+1}^n) \in rad(Q)$  for some $i \in \{1,\cdots,u\}$ or   $g(f(x_1,a_n,0^{(m-2)}),\cdots,f(x_u,a_u,0^{(m-2)}),x_{u+1}^{j-1},s, x_{j+1}^n) \in rad(Q)$  or $g(s,x_j,1_A^{(n-2)}) \in Q$   for some $j \in \{u+1,\cdots,n\}$. This implies that $g(s,x_i,1_A^{(n-2)}) \in Q$ and $g(x_1^{i-1},s,x_{i+1}^n) \in  rad(Q)$ for some $i \in \{1,\cdots,n\}$ which is a contradiction. Thus $g(x_1,\cdots,\widehat{x_{j_1}}, \cdots,\widehat{x_{j_2}},\cdots, \widehat{x_{j_u}}, \cdots, Q^{(u)})=0$ for each $j_1,\cdots,j_u \in \{1,\cdots,n\}$.  
\end{proof}
\begin{theorem} \label{28}
Assume that $S \subseteq A$ is an $n$-ary multiplicative set and  $Q$ is a strongly weakly $n$-ary $S$-primary hyperideal of $A$. If  $g(Q^{(n)}) \neq 0$, then  $Q$ is an $n$-ary $S$-primary hyperideal of $A$. 
\end{theorem}
\begin{proof}
Suppose that $Q$ is a strongly weakly $n$-ary $S$-primary  hyperideal of $A$ and $s$ is a strongly weakly $S$-primary element of $Q$.  Let $x_1^n \in A$ such that $g(x_1^n) \in Q$. If $0 \neq g(x_1^n) \in Q$, then we get $g(s,x_i,1_A^{(n-2)}) \in Q$ or $g(x_1^n,s,x_{i+1}^n) \in rad(Q)$ for some $i \in \{1,\cdots,n\}$ which implies $Q$ is an $n$-ary $S$-primary hyperideal of $A$.  Suppose  that $g(x_1^n)=0$. If  $g(x_1,\cdots,\widehat{x_{j_1}}, \cdots,\widehat{x_{j_2}},\cdots, \widehat{x_{j_u}}, \cdots, Q^{(u)}) \neq 0$. Then we conclude that  $g(s,x_i,1_A^{(n-2)}) \in Q$ or $g(x_1^n,s,x_{i+1}^n) \in rad(Q)$ for some $i \in \{1,\cdots,n\}$ in view of Theorem \ref{27}. Let $g(x_1,\cdots,\widehat{x_{j_1}}, \cdots,\widehat{x_{j_2}},\cdots, \widehat{x_{j_u}}, \cdots, Q^{(u)})= 0$ for all $u \in \{1,\cdots,n-1\}$.
 Since $ 0 \neq g(Q^{(n)})$,  there exist $a_1^n \in Q$ such that $g(a_1^n) \neq 0$. Therefore we have $0 \neq g(f(x_1,a_1,0^{(m-2)}), \cdots,f(x_n,a_n,0^{(m-2})) \subseteq Q$ by Theorem \ref{27}. Hence  $g(s,f(x_i,a_i,0^{(m-2)}) \subseteq Q$ and so $f(g(s,x_i,1_A^{(n-2)}),g(s,a_i,1_A^{(n-2)}),0^{(m-2)}) \subseteq Q$     
 which implies $g(s,x_i,1_A^{(n-2)}) \in Q$ as $g(s,a_i,1_A^{(n-2)}) \in Q$ or we conclude that  $g(f(x_1,a_1,0^{(m-2)}),\cdots,\widehat{f(x_{i},a_{i},0^{(m-2)})},s,\cdots,f(x_n,a_n,  0^{(m-2)})) \subseteq rad(Q)$ which means  $g(x_1^{i-1},s,x_{i+1}^n) \in rad(Q)$ by Theorem \ref{27}. Consequenly,  $Q$ is an $n$-ary $S$-primary hyperideal of $A$.
\end{proof} 
We give the following theorem as a result of Theorem \ref{28}.
\begin{theorem}
Assume that $S \subseteq A$ is an $n$-ary multiplicative set and  $Q$ is a strongly weakly $n$-ary $S$-primary hyperideal of $A$. If   $Q$ is not an $n$-ary $S$-primary hyperideal of $A$, then $rad(Q)=rad(0)$. 
\end{theorem}
 
Assume that $M $ is a non-empty set. Recall from \cite{Anvariyeh} that a triple $(M, f^{\prime}, g^{\prime})$ is called an $(m, n)$-hypermodule over a commutative Krasner $(m,n)$-hyperring $(A,f,g)$     if  $(M, f^{\prime})$ is an $m$-ary hypergroup and  the map 
\[g^{\prime}:\underbrace{A \times \cdots \times A}_{n-1} \times M\longrightarrow 
P^*(M)\]
statisfied the following conditions:
\begin{itemize} 
\item[\rm{(1)}]~ $g^{\prime}(a_1^{n-1},f^{\prime}(x_1^m))=f^{\prime}(g^{\prime}(a_1^{n-1},x_1),\cdots,g^{\prime}(a_1^{n-1},x_m))$
\item[\rm{(2)}]~ $g^{\prime}(a_1^{i-1},f(b_1^m),a_{i+1}^{n-1},x)=f^{\prime}(g^{\prime}(a_1^{i-1},b_1,a_{i+1}^{n-1},x),\cdots,g^{\prime}(a_1^{i-1}
b_m,a_{i+1}^{n-1},x))$
\item[\rm{(3)}]~ $g^{\prime}(a_1^{i-1},g(a_i^{i+n-1}),a_{i+m}^{n+m-2},x)=
g^{\prime}(a_1^{n-1},g^{\prime}(a_m^{n+m-2},x))$
\item[\rm{(4)}]~ $ 0=g^{\prime}(a_1^{i-1},0,a_{i+1}^{n-1},x)$.
\end{itemize}

In the following theorem, we consider Nakayama$^,$s Lemma for a strongly weakly $n$-ary $S$-primary hyperideal.
\begin{theorem} \label{29}
Assume that  $(M,f^{\prime},g^{\prime})$ is an $(m,n)$-hypermodule over $(A,f,g)$ and  $Q$ is a strongly weakly $n$-ary $S$-primary hyperideal of $A$ that is not $n$-ary $S$-primary. If   $M=g^{\prime}(Q,1_A^{(n-2)},M)$, then  $M=\{0\}$. 
\end{theorem}
\begin{proof}
Let  $Q$ be a strongly weakly $n$-ary $S$-primary hyperideal of $A$ but is not $n$-ary $S$-primary and $M=g^{\prime}(Q,1_A^{(n-2)},M)$ where $(M,f^{\prime},g^{\prime})$ is an $(m,n)$-hypermodule over $(A,f,g)$. By Theorem \ref{28},  we conclude that $g^{\prime}(g(Q^{(n)}),1_A^{(n-2)},M)=\{0\}$. On the other hand 
$g^{\prime}(g(Q^{(n)}),1_A^{(n-2)},M)=g^{\prime}(g(Q^{(n-1)},1_A),1_A^{(n-2)},g^{\prime}(Q,1_A^{(n-2)},M))=g^{\prime}(g(Q^{(n-1)},1_A),1_A^{(n-2)},M)=\cdots=g^{\prime}(Q,1_A^{(n-2)},g^{\prime}(Q,1_A^{(n-2)},M))=g(Q,1_A^{(n-2)},M)=M$. Then we get the result that $M=0$. 
\end{proof}


\end{document}